\documentclass[12pt,longbibliography]{article}

\usepackage[utf8]{inputenc} % Allow direct use of accents such as á é ñ.
\usepackage[T1]{fontenc}

\usepackage{graphicx}
\usepackage{amsmath}
\usepackage{amsfonts}
\usepackage{mathrsfs}           % \mathscr font.
\usepackage{amsthm}
\usepackage{amssymb}
\usepackage{color}
\usepackage{enumerate}

\usepackage[a4paper, margin=2.7cm]{geometry}

\usepackage[colorlinks=true,linkcolor=blue,citecolor=blue,urlcolor=blue,breaklinks]{hyperref}

\usepackage{bbm}

\usepackage{breakurl}
\usepackage{url}
\def\C{\mathbb{C}}
\def\N{\mathbb{N}}
\def\R{\mathbb{R}}

\def\d{\,\mathrm{d}}

\newtheorem{thm}{Theorem}[section]

\newtheorem{lem}[thm]{Lemma}
\newtheorem{prp}[thm]{Proposition}

\theoremstyle{definition}
\newtheorem{dfn}[thm]{Definition}
\theoremstyle{remark}
\newtheorem{rem}[thm]{Remark}
\theoremstyle{example}

\title{Intrinsic Ultracontractivity for a class of Schrödinger Semigroups in $\mathrm{L}^{2}\left( \R^{n} \right)$
by Logarithmic Sobolev inequalities}

\author{Christoph Schwerdt$^1$, Alexander Mill$^2$ and Dirk Hundertmark$^3$}
\date{%
	$^1$ Institute of Mathematics, University of Rostock,\\
	Ulmenstra\ss e 69, 18 057 Rostock, Germany; \\%
	\textit{E-mail address}: \texttt{\href{mailto:christoph.schwerdt@uni-rostock.de}{christoph.schwerdt@uni-rostock.de}}\\
	\ \\
	$^2$ Institute of Mathematics, University of Rostock,\\
	Ulmenstra\ss e 69, 18 057 Rostock, Germany; \\%
	\textit{E-mail address}: \texttt{\href{mailto:alexander.mill@uni-rostock.de}{alexander.mill@uni-rostock.de}}\\
	\ \\
	$^3$ Department of Mathmatics, Karlsruhe Institute of Technology (KIT), \\
	76 128 Karlsruhe, Germany; \\%
	\textit{E-mail address}: \texttt{\href{mailto:dirk.hundertmark@kit.edu}{dirk.hundertmark@kit.edu}}\\
	\ \\
    \today
}

\begin{document}

\maketitle

\begin{abstract}
In the first part of this article we present a growth condition on the potential $q$ in the Schrödinger operator 
$H=-\Delta + q(x)$ in $\mathrm{L}^{2}\left( \R^{n} \right)$ 
that implies \textit{Rosen inequalities} for the ground state $\varphi$ of $H$, i.e.
$$
\forall \varepsilon > 0 \exists \gamma(\varepsilon) > 0 \ : \ - \ln\left( \varphi(x) \right) \leq \varepsilon q(x) + \gamma(\varepsilon).
$$ 
While these inequalities are not particularly interesting in themselves, they
offer Logarithmic Sobolev inequalities which are absolutely essential to prove an \textit{intrinsic ultracontractivity} of 
the associated Schrödinger semigroup $\mathrm{e}^{-tH}$, i.e.
$$
\forall t>0 \exists C_{t} > 0 \ : \ \left| \mathrm{e}^{-tH} u (x) \right| \ \leq \ C_{t} \varphi(x) \| u \|_{2}
$$
holds for every $u \in \mathrm{L}^{2}\left( \R^{n} \right)$ almost everywhere in $\R^{n}$ which we prove in the second 
part of this article. For proving Rosen inequalities we focus on solving a radial Schrödinger inequality and use
\textit{Agmon's version of the comparison principle} and \textit{Young's inequality for increasing functions}.
We follow the classic method proving intrinsic ultracontractivity of $\mathrm{e}^{-tH}$ by using
\textit{weighted Sobolev function spaces}, \textit{weighted Schrödinger semigroups} and \textit{Logarithmic Sobolev inequalities}.\\
\end{abstract}

\tableofcontents

\section{Introduction}

\subsection{Intrinsic ultracontractivity in $\mathrm{L}^{2}(\R^{n})$}

The Hamiltonian corresponding to a quantum system is given by the formal operator
\begin{equation}
H =  -\Delta + q(x)
\end{equation}
in $\mathrm{L}^{2}(\R^{n})$ in $n \geq 3$ dimensions with $q \in \mathrm{L}^{1}_{loc}\left( \R^{n} \right)$. We focus on
potentials $q$ being non-negative almost everywhere in $\R^{n}$ such that there exists a unique strictly positive 
eigenfunction $\varphi$ corresponding to the lowest eigenvalue $E_{0}$ of $H$ which is the system's \textit{ground state}.
We present a growth condition of $q$ which implies \textit{Rosen inequalities}, i.e.
\begin{equation}
\forall \varepsilon > 0 \ \exists \gamma(\varepsilon) > 0 \ : \ - \ln \left( \varphi(x) \right) \ \leq \ \varepsilon q(x) + \gamma(\varepsilon)
\end{equation}
for almost every $x \in \R^{n}$. These are not particularly interesting in themselves, but imply 
\textit{Logarithmic Sobolev inequalties} which are absolutely essential to prove 
\textit{intrinsic ultracontractivity} of the Schrödinger semigroup $\mathrm{e}^{-tH}$, i.e.
\begin{equation}
\forall t > 0 \exists C_{t} > 0 \ : \ \left| \mathrm{e}^{-tH}u (x)  \right| \ \leq \ C_{t} \varphi(x) \| u \|_{2}
\end{equation}
for every $u \in \mathrm{L}^{2}(\R^{n})$ almost everywhere in $\R^{n}$. In this case the asymptotic behaviour 
of $\mathrm{e}^{-tH}u$ is dominated by the ground state $\varphi$ at every time $t>0$. In particular, for a normed 
eigenfunction $v$ of $H$ we conclude to 
\begin{equation}
\exists C_{\lambda} > 0 \ : \ |v(x)| \leq C_{\lambda} \varphi(x)
\end{equation}
almost everywhere in $\R^{n}$ where $\lambda$ is the associated eigenvalue of $v$. Mind that in quantum physics
eigenfunctions are characterized as a probability density function of an electron's position at a certain energy level 
$\lambda$ of the system $H$.
But apart from quantum mechanics we can interpret intrinsic ultracontractivity as perturbation result of the generator $-\Delta$ 
of the free Schrödinger semigroup given by
\begin{equation}
\mathrm{e}^{t\Delta}u (x) = \frac{1}{(4 \pi t)^{\frac{n}{2}}} 
\int_{\R^{n}} \mathrm{e}^{-\frac{|x-y|^{2}}{4t}} u(y) \ \mathrm{d}y
\end{equation} 
for $u \in \mathrm{L}^{2}(\R^{n})$. Due to the Gaussian integral kernel $\mathrm{e}^{t\Delta}$ is 
a contraction in $\mathrm{L}^{p}(\R^{n})$ for every $p \in [1,\infty]$ and further maps $\mathrm{L}^{p}(\R^{n})$ into 
$\mathrm{L}^{q}(\R^{n})$ continuously for $1 \leq p < q \leq \infty$.
For the case of $H=-\Delta + q(x)$ we cite the following characterization of 
intrinsic ultracontractivity as Lemma 4.2.2 from page 110 in \cite{Davies07}. \\

\begin{lem} The Schrödinger semigroup $\mathrm{e}^{-tH}$ is intrinsic ultracontractive if and only if both of the following 
conditions are satisfied for every time $t > 0$
\begin{enumerate}[i)]
\item $\mathrm{e}^{-tH}u (x) = \int_{\R^{n}} \ k(t,x,y) u(y) \d y$ almost everywhere in $\R^{n}$ for 
$u \in \mathrm{L}^{2}(\R^{n})$ 
\item $\exists C_{t} > 0 \ : \ 0 \leq k(t,x,y) < C_{t} \ \varphi(x) \ \varphi(y)$ almost everywhere in $\R^{2n}$\\
\end{enumerate}
\end{lem}

\subsection{Historical context}

First results of intrinsic hypercontractivity and intrinsic ultracontractivity of Schrödinger semigroups date back to the 1970s
as mentioned on page 336 in \cite{DaviesSimon84}. However until the mid-1980s 
there was a folk belief that intrinsic ultracontractivity of Schrödinger semigroups in $\mathrm{L}^{2}(\R^{n})$ would not occur. 
A simple example in that regard is the quantum harmonic oscillator $H=-\Delta + |x|^{2}$ in $\mathrm{L}^{2}(\R^{n})$ where it 
is very easy to see that $\mathrm{e}^{-tH}$ is not intrinsic ultracontractive by 
comparing the asymptotical behaviour of eigenfunctions to the ground state. 
However, in \cite{DaviesSimon84}  intrinsic ultracontractivity of $\mathrm{e}^{-tH}$ is shown
for potentials close to $|x|^{2}$ like in
$H=-\Delta + |x|^{\beta}$ or in $H=-\Delta + |x|^{2} \left( \ln(|x|+2)\right)^{\beta}$ for any $\beta > 2$. 
Even more examples are given in \cite{Davies07} which provides 
an in-depth study on intrinsic ultracontractivity of Schrödinger semigroups in $\mathrm{L}^{2}(\R^{n})$.
For bounded domains $D \subset \R^{n}$ with a sufficiently smooth boundary Banuelos  \cite{BANUELOS1991} proved 
intrinsic ultracontractivity of Schrödinger semigroups in $\mathrm{L}^{2}(D)$ in 1990. 
In the 2000s B. Alziary and P. Takáč returned to the original studies of intrinsic ultracontractivity 
in $\mathrm{L}^{2}(\R^{n})$ in \cite{AlziaryTakac09}. A growth condition of $q$ was presented which implies an 
intrinsic ultracontractivity of $\mathrm{e}^{-tH}$. For the frist time the Schrödinger potential $q$ itself did not have to be 
radial anymore. Instead, 
it is squeezed inbetween two radial bounding auxiliary potentials such that
\begin{equation}\label{Takac_boundary}
\left( \int_{r_{0}}^{|x|} Q(t)^{\frac{1}{2}} \d t \right) \ P \left[ \ln \left( \int_{r_{0}}^{|x|} Q(t)^{\frac{1}{2}} \d t  \right) \right]
\leq q(x) \leq Q\left( |x| \right)
\end{equation}
holds for every $|x| \geq r_{0}$ for auxiliary functions $P$ and $Q$ where $Q$ satisfies
\begin{equation}\label{condition_Q_Hartman_Wintner}
\int_{r_{0}}^{\infty} \left| \frac{\d}{\d t} \left( Q(t)^{\frac{1}{2}} \right)  \right|^{\gamma} Q(t)^{\frac{1}{2}} \d t < \infty
\end{equation}
for a constant $\gamma \in (1,2]$ and $P$ satisfies $\int_{0}^{\infty} P(t)^{-1} \d t < \infty$. Examples of $Q$ are given on 
page 4105 by $|x|^{2+\delta}$, $|x|^{2} \left( \ln |x| \right)^{2+\delta}$, 
$|x|^{2}  \left( \ln |x| \right)^{2} \left( \ln \ln |x| \right)^{2+\delta}$ et cetera for $\delta > 0$.

\subsection{Motivation of this article}

We follow \cite{AlziaryTakac09} in terms of using a comparison principle, Young's inequality for increasing functions and
the classical use of Logarithmic Sobolev inequalities to prove an intrinsic ultracontractivity of $\mathrm{e}^{-tH}$. 
However we won't need condition (\ref{condition_Q_Hartman_Wintner}). Instead we demand 
$$
\frac{Q^{\prime}(r)}{Q(r)^{\frac{3}{2}}} \ \to \ 0
$$
for $r \to \infty$ which is easier to verify for concrete choices of $Q$. Furthermore we focus 
on the radial Schrödinger inequality 
\begin{equation}\label{radial_Schroedinger_inequality}
- \psi^{\prime \prime} (r) - \frac{n-1}{r} \psi^{\prime}(r) + Q(r)\psi(r) \ \leq \ 0 
\end{equation}
in contrast to the radial Schrödinger equation (36) on page 4112 in \cite{AlziaryTakac09}. It is much easier to find 
solutions to (\ref{radial_Schroedinger_inequality}) which are sufficient in terms of Rosen inequalities. Therefore
our auxiliary function $\psi$ is much simpler than Lemma 4.4 in \cite{AlziaryTakac09}.\\   
Furthermore, we present a specific but reasonable choice of $P$ in (\ref{Takac_boundary}) to guarantee a better 
understanding. We use the auxiliary function
$f \colon \R \to [0,\infty)$ for $P$ defined by
$$
f(q)=\left\{\begin{array}{ll} 
	f_{k,m-1}(q), & q \geq r_{0} \\
         f_{k,m-1}\left( r_{0} \right)\mathrm{e}^{\frac{q}{r_{0}}-1}, & q < r_{0}
\end{array}\right. 
$$
where $r_{0} = \ln \left( \int_{0}^{R_{m}} Q(t)^{\frac{1}{2}} \ \mathrm{d}t \right) > 0$ and
$
f_{k, m} (t) = \left( \ln^{(m)}(t) \right)^{k} \ \prod_{p=0}^{m-1} \ln^{(p)}(t)
$ 
for $k>1$, $m \in \N$ and $t > R_{m}$ with $R_{m} > 0$ chosen large that $\ln^{(m)}(t) > 0$ is true.\footnote{
Note that $\ln^{(0)}(t) = t$ and $\ln^{(p)}(t) =\underbrace{\ln(\ln(\dots(\ln(t)) \dots))}_{p-\text{times}}$}\\

\section{Definitions and preparations for Rosen inequalities}\label{definitions_and_preparations}

We consider dimensions $n \geq 3$ and write $\mathrm{L}^{2} \left( \R^{n} \right)$ 
for the set of complex-valued and measurable functions $u$ on $\R^{n}$ with $\int_{\R^{n}} |u(x)|^{2} \ \mathrm{d}x < \infty$.
For real-valued functions in $\mathrm{L}^{2}\left( \R^{n} \right)$ we explicitly write $\mathrm{L}^{2} \left( \R^{n}, \R \right)$.
Throughout this article the Schödinger potential $q \colon \R^{n} \to \R$ is continuous and 
non-negative almost everywhere in $\R^{n}$ with $q(x) \to \infty$ for $|x| \to \infty$.\\

We define a subspace $D(h)$ as the form domain of $h$ by
\begin{equation}
D(h) = \left\{  u \in H^{1}\left( \R^{n} \right) \ | \ q^{\frac{1}{2}}u \in  \mathrm{L}^{2}(\R^{n}) \right\}  \subset 
\mathrm{L}^{2}\left( \R^{n} \right)
\end{equation}
where $H^{1}\left( \R^{n} \right)$ denotes the Sobolev space of weakly differentiable functions in 
$\mathrm{L}^{2} \left( \R^{n} \right)$
whose weak derivatives are contained in $\mathrm{L}^{2} \left( \R^{n} \right)$. Please mind that 
$C_{c}^{\infty}\left( \R^{n} \right) \subseteq D(h)$ is satisfied. 
For $u,v \in D(h)$ we define a sesquilinear form $h$ by
\begin{equation}
h(u,v) = \left\{ \sum_{j = 1}^{n} \langle \partial_{j} u , 
\partial_{j} v \rangle \right\} + \langle q^{\frac{1}{2}}u, q^{\frac{1}{2}}v \rangle
\end{equation}
where $\langle \cdot , \cdot \rangle$ denotes the inner product of $\mathrm{L}^{2}(\R^{n})$. 
The form $h$ induces a norm $\| u \|_{h} = \sqrt{\langle u,u \rangle_{h}}$ with
$\langle u,v \rangle_{h} = h(u,v) + \langle u,v \rangle$ for  $u,v \in D(h)$ which turns $D(h)$ into a Hilbert space.
A \textit{Schrödinger operator} $H$ is defined as the associated operator to $h$ with a domain $D(H)$ given by
\begin{equation}
D(H) = 
\left\{  u \in  D(h) \ | \  
\exists v \in \mathrm{L}^{2} \left( \R^{n} \right) \ : \  h(u,w) = \langle v, w \rangle \ \text{for every } 
w \in D(h) \right\}.
\end{equation}
Hence $H$ satisfies the equation 
\begin{equation}
\langle Hu, w \rangle = h(u,w)
\end{equation}
for $u \in D(H)$ and every $w \in D(h)$. Furthermore, $H$ is self-adjoint in $\mathrm{L}^{2} \left( \R^{n} \right)$
such that $(0,\infty)$ is contained in the resolvent set $\rho (H)$. The latter is shown by the use of Riesz's representation 
theorem using the completeness of $D(h)$ with respect to $\| \cdot \|_{h}$. 
The Lumer-Phillips theorem infers that $-H$ is the generator of a $C_{0}$-semigroup 
$\left\{ \ \mathrm{e}^{-tH} \ | \ t \geq 0 \ \right\}$
of contractions in $\mathrm{L}^{2} \left( \R^{n} \right)$ which we call a \textit{Schrödinger Semigroup} in 
$\mathrm{L}^{2} \left( \R^{n} \right)$. Furthermore the spectrum 
$\sigma(H)$ consists only of eigenvalues with a \textit{ground state energy} $E_{0} = \min \sigma(H) \in [0,\infty)$ being simple
and a corresponding eigenfunction $\varphi$ being strictly positive almost everywhere in $\R^{n}$. We call $\varphi$ the 
\textit{ground state} of $H$. Now, let us show that $\varphi$ is a continuous function in $\R^{n}$.\\

\begin{lem}\label{ground_state_continuous}
The ground state $\varphi$ of $H$ has a continuous representative on $\R^{n}$.\\
\end{lem} 

\begin{proof}
We use Theorem 8.22 on page 200 in \cite{GilbargTrudinger01} and define $L= \Delta - q(x) + E_{0}$ 
in $\mathrm{L}^{2}\left( B_{R}(0) \right)$ for any given radius $R>0$
in the same manner we defined $H$ in $\mathrm{L}^{2}\left( \R^{n} \right)$.
Then $\varphi \in H^{1}\left( B_{R}(0) \right)$ is a weak
solution to $Lu = 0$ in $B_{R}(0)$. By Theorem 8.22 we treat $\varphi$ as a continuous function in $B_{R/2}(0)$ and since
$R > 0$ was chosen arbitrarily the claim is implied.\\
\end{proof}

Furthermore we show that the ground state $\varphi$ is 
strictly positive everywhere in $\R^{n}$. Due to the continuity of $\varphi$ this implies 
$$
\min_{x \in \overline{B_{R}(x_{0})}} \varphi(x) = \delta_{x_{0},R} > 0
$$ 
for every $x_{0} \in \R^{n}$ and $R > 0$ where $B_{R}(x_{0})$ denotes a ball of radius $R$ in $\R^{n}$ centered at $x_{0}$.\\

\begin{lem}\label{strict_positivity_ground_state}
The ground state $\varphi$ is strictly positive everywhere in $\R^{n}$.\\
\end{lem}

\begin{proof}
The ground state $\varphi$ is strictly positive almost everywhere in $\R^{n}$. We use the Harnack inequality
in form of Theorem 8.20 on page 199 in \cite{GilbargTrudinger01} to argue that $\varphi(x) > 0$ is true for every $x \in \R^{n}$.\\

We prove by contradiction. Suppose there is a $x_{0} \in \R^{n}$ with $\varphi(x_{0})=0$.
We define the elliptic operator $L$ as mentioned in Lemma \ref{ground_state_continuous} 
in $\mathrm{L}^{2}\left( B_{R}(x_{0}) \right)$ for any fixed radius $R>0$. Then
$L$ meets all requirements of Theorem 8.20
in \cite{GilbargTrudinger01} with $L\varphi \equiv 0$ in $B_{R}(x_{0})$. Hence
\begin{equation} 
\exists C > 0 \ : \ \sup_{x \in B_{R/4}(x_{0})} \varphi(x) \ \leq \ C \inf_{x \in B_{R/4}(x_{0})} \varphi(x)
\end{equation}
is implied. Please mind that Theorem 8.20 originally refers to the essential infimum and essential supremum which is equal to
the regular infimum and supremum respectively due to the continuity of $\varphi$. Then 
$ \inf_{x \in B_{R/4}(x_{0})} \varphi(x) = \varphi(x_{0}) = 0$ implies $\sup_{x \in B_{R/4}(x_{0})} \varphi(x) \leq 0$. But that 
contradicts $\varphi > 0$ almost everywhere in $\R^{n}$. Therefore, such an $x_{0}$ does not exist and  $\varphi(x) > 0$ 
is true for every $x \in \R^{n}$.\\
\end{proof}

\begin{rem} \
\begin{enumerate}[i.)]
\item Notice that the continuity of $\varphi$ is not necessarily needed in Section \ref{Proofs_for_Rosen_Inequalities}. In particular,
we only need a lower boundary $C_{R} > 0$  of $\varphi$ on a ball $B_{R}(0)$. However, continuity of $\varphi$ is much more 
convenient and is also implied as we have seen.
\item For much more information on Schrödinger forms and operators please consider \cite{Leinfelder81} by  
H. Leinfelder and C. G. Simader.
\item The statement on the ground state energy $E_{0}$ of $H$ is taken from  Theorem 10.11 on page $236$ in \cite{Teschl09} and
holds due to an expansion of the Perron-Frobenius theory to bounded and positivity improving operators in infinite dimensional 
vector spaces.
\end{enumerate}
\end{rem}

\section{A class of $q$ implying Rosen Inequalities} \label{Main_theorems_Rosen}

We present a growth condition of $q \colon \R^{n} \to [0,\infty)$ in 
$H = -\Delta + q(x)$ that implies Rosen inequalities for the ground state $\varphi \in D(H)$, i.e.
\begin{equation}\label{Rosen}
\forall \varepsilon > 0 \ \exists \gamma(\varepsilon) > 0 \ \forall x \in \R^{n} \ : \ 
-\ln \left( \varphi(x) \right) \ \leq \ \varepsilon q(x) + \gamma(\varepsilon).\\
\end{equation}

For $k>1$ and $m \in \N$ we define  
\begin{equation}
f_{k, m} (t) = \left( \ln^{(m)}(t) \right)^{k} \ \prod_{p=0}^{m-1} \ln^{(p)}(t)
\end{equation} 
for $t > R_{m}$ with $R_{m} > 0$ chosen large that $\ln^{(m)}(t) > 0$ is true.\footnote{Note that $\ln^{(0)}(t) = t$ and
$\ln^{(p)}(t) =\underbrace{\ln(\ln(\dots(\ln(t)) \dots))}_{p-\text{times}}$}\\

\begin{thm} \label{Main_Thm_2}
Let $Q \colon [0,\infty) \to (0,\infty)$ be monotone increasing with $r^{2} < Q(r)$. Furthermore let 
$m \in \N$, $k>1$ and $R_{m} > 0$ such that
\begin{enumerate}[i.)]
\item $Q$ is differentiable in $(R_{m}, \infty)$,
\item $\forall  r \geq R_{m} \ : \ \ln^{(m)}\left(  \int_{0}^{r} Q(t)^{\frac{1}{2}} \ \mathrm{d}t  \right) \ > \ 0$ and
\item $\forall  r \geq R_{m} \: \ 0 \ < \ f_{k, m-1}\left( \ln \left( Q(r) \right) \right)rQ(r)^{-\frac{1}{2}} \ < \ 1$.\\
\end{enumerate} 
Also let $Q^{\prime}(r)Q(r)^{-\frac{3}{2}} \to 0$ for $r \to \infty$ and $d \in (0,1]$. Then for every continuous 
potential $q \colon \R^{n} \to [0,\infty)$ satisfying
\begin{equation}
d \left( \int_{0}^{ |x|} Q(t)^{\frac{1}{2}} \ \mathrm{d}t \right) \
f_{k, m-1} \left( \ln \left( \int_{0}^{ |x|} Q(t)^{\frac{1}{2}} \ \mathrm{d}t \right) \right) \ \leq q(x) \ \leq \ Q \left( |x| \right)
\end{equation}
for $|x| \geq R_{m}$, the ground state $\varphi$ of $H=-\Delta+q(x)$ satisfies Rosen inequalities.\\
\end{thm}

\begin{rem} \label{Remark 1} \
\begin{enumerate}[a)]

\item Notice that $\int_{0}^{r} Q(t)^{\frac{1}{2}} \ \mathrm{d}t > \int_{0}^{r} t \d t = \frac{1}{2}r^{2}$ and 
$$
f_{k, m-1} \left( \ln \left( \int_{0}^{r} Q(t)^{\frac{1}{2}} \ \mathrm{d}t \right) \right) \geq 
\ln \left( \int_{0}^{r} Q(t)^{\frac{1}{2}} \ \mathrm{d}t \right) > \frac{2}{d} 
$$
are satisfied for large $r$. Therefore $q(x) > |x|^{2}$ is true for $|x|$ large.

\item We use the auxiliary function $f_{k, m-1}$ in the lower bound of $q$ to include $\varepsilon$ in the Rosen inequalities
of $\varphi$ without adding much growth. However, $k > 1$ and $\prod_{p=0}^{m-1} \ln^{(p)}(t)$ in $f_{k,m-1}$ 
guarantee 
$$
\int_{R_{m}}^{\infty} \frac{1}{f_{k, m} (t)} \d t < \infty
$$
which is needed to prove Theorem \ref{IU_of_e^{-tH}}.\\

\end{enumerate}
\end{rem}

We give examples of possible bounding potentials $Q$ which are close to $r^{2}$. Please mind that these are taken 
from Example 2.5 on page 4105 in \cite{AlziaryTakac09} and all meet the requirements of Theorem \ref{Main_Thm_2}

\begin{lem} \label{Lemma_1} \ \\
We present auxiliary functions satisfying the requirements of Theorem \ref{Main_Thm_2} for large $r$:
\begin{enumerate}[a)]
\item $Q_{1}(r) = r^{\alpha}$ for $\alpha > 2$,
\item $Q_{2}(r) = \left(\ln(r) \right)^{\alpha} r^{2} $ for $\alpha > 2$,
\item $Q_{3}(r) = \left( \ln\left( \ln(r)\right) \right)^{\alpha} \left(\ln(r) \right)^{2} r^{2}$ for $\alpha > 2$ and
\item $Q_{4}(r) = \left( \ln^{(l)}(r) \right)^{\alpha} \prod_{p=0}^{l-1} \left( \ln^{(p)}(r) \right)^{2}$ for $\alpha > 2$ and any 
chosen $l \in \N$.\\
\end{enumerate}
\end{lem}

\section{Proofs for Rosen Inequalities} \label{Proofs_for_Rosen_Inequalities}

\subsection{Preliminary} \label{Preliminary}

First, we determine conditions on the upper bounding potential $Q$ such that the 
strictly positive ground state $\varphi$ of $H = -\Delta + q(x)$ could possibly satisfy Rosen inequalities, i.e.
\begin{equation*}
\forall \varepsilon > 0 \ \exists \gamma(\varepsilon) > 0 \ \forall x \in \R^{n} \ : \ 
-\ln \left( \varphi(x) \right) \ \leq \ \varepsilon q(x) + \gamma(\varepsilon).\\
\end{equation*}
Let $Q$ be an upper bounding function such that $q(x) \ \leq \ Q \left( |x| \right)$
holds for $ |x| \geq R$ and let $\psi \in \mathrm{L}^{2} \left( \R^{n} \right)$ be a strictly positive solution of
\begin{equation}\label{psi_subsolution}
0 \ \geq \ \left( - \Delta \psi \right) (x) + Q\left(  |x| \right) \psi(x) \ \geq \ \left( - \Delta \psi \right) (x) + q\left( x \right) \psi(x)
\end{equation}
for every $|x| \geq R$. Then $\psi$ is a subsolution of $H$ in $\Omega_{R} = \{ x \in \R^{n} \ : \ |x| > R \}$ at the ground state 
energy level $E_{0} = \min \left( \sigma(H) \right) \geq 0$ by Definition \ref{Definition_sub_supersolution}.
The ground state $\varphi$, on the other hand, is a supersolution 
of $H$ in $\Omega_{R}$ at $E_{0}$. Using Theorem \ref{comparision_principle} and Remark \ref{Remark_comparision_principle} 
there exists a constant $c > 0$ such that 
$$
\psi(x)\leq c \varphi(x)
$$
holds for every  $x \in \Omega_{R}$. We use Lemma \ref{ground_state_continuous} and 
Lemma \ref{strict_positivity_ground_state} to conclude that
$$
\varphi(x) \geq \min_{y \in \overline{B_{R}(0)}} \varphi(y) := \delta_{0} > 0
$$
is satisfied for every $|x| \leq R$. Therefore,
\begin{equation}\label{Rosen_psi_phi}
-\ln \left( \varphi(x) \right) \ = \ \ln \left( \frac{1}{\varphi(x)} \right) \ \leq \ \ln \left( \frac{1}{\psi(x)} \right) + C
\end{equation}
is implied for every $x \in \R^{n}$ for an appropriate constant $C > 0$. If $\psi$ additionally satisfies 

\begin{equation}\label{psi_Rosen}
 \ln \left( \frac{1}{\psi(x)} \right) \leq \varepsilon q(x) + \gamma(\varepsilon)
\end{equation}
for every $x \in \R^{n}$, then Rosen inequalities of $\varphi$ are implied. 
Next, let us consider the radial Schrödinger inequality

\begin{equation} \label{Radial_Schrödinger_inequality}
\psi^{\prime \prime}(r) + \frac{n-1}{r} \psi^{\prime}(r) \ \geq \ Q(r)\psi(r)
\end{equation}
for strict positive and twice differentiable functions $\psi \in \mathrm{L}^{2}\left( (0,\infty), r^{n-1} \mathrm{d}r \right)$.
We define
$$
\psi(r) = \exp \left( - \sqrt{2}\int_{0}^{r} Q(t)^{\frac{1}{2}} \ \mathrm{d}t   \right)
$$
for $r \geq 0$. Furthermore $Q(r) \to \infty$ and $Q^{\prime}(r)Q(r)^{-\frac{3}{2}} \to 0$ for $r \to \infty$ imply
$$
-\frac{1}{2} < \frac{Q^{\prime}(r)}{Q(r)^{\frac{3}{2}}} + \frac{n-1}{rQ(r)^{\frac{1}{2}}} < \frac{1}{2}  
$$
for every $r \geq R$ and $R$ sufficiently big. Then
\begin{align*}
& \psi^{\prime \prime}(r) + \frac{n-1}{r} \psi^{\prime}(r) = 
Q(r) \psi(r) \left(  2 -  \frac{Q^{\prime}(r)}{\sqrt{2}Q(r)^{\frac{3}{2}}}  - \frac{\sqrt{2}(n-1)}{rQ(r)^{\frac{1}{2}}}  \right) \\
& \geq Q(r) \psi(r) \left(  2 - \sqrt{2} \left( \frac{Q^{\prime}(r)}{Q(r)^{\frac{3}{2}}} + \frac{n-1}{rQ(r)^{\frac{1}{2}}} \right) \right) \\
& \geq Q(r) \psi(r) \left(  2 - \frac{1}{\sqrt{2}} \right) \geq Q(r) \psi(r) 
\end{align*}
follows for every $r \geq R$. We conclude to
\begin{equation}\label{Rosen_phi}
-\ln \left( \varphi(x) \right) \ = \ \ln \left( \frac{1}{\varphi(x)} \right) \ \leq \ \ln \left( \frac{1}{\psi(|x|)} \right) + C \ = \
\sqrt{2} \int_{0}^{|x|} Q(t)^{\frac{1}{2}} \ \mathrm{d}t + C
\end{equation}
for every $x \in \R^{n}$ with $|x| \geq R$. Using Lemma \ref{ground_state_continuous} and 
Lemma \ref{strict_positivity_ground_state} we can choose $C$ sufficiently large such that (\ref{Rosen_phi}) is true for every 
$x \in \R^{n}$. This solution $\psi$ of the radial Schrödinger inequality 
(\ref{Radial_Schrödinger_inequality}) is the
foundation of our approach for Rosen inequalities of $\varphi$. \\

\subsection{Proof of Theorem \ref{Main_Thm_2}} \label{Proof_Main_Thm_2}

\begin{enumerate}[a)]

\item Notice that
$$
f_{k,m-1} \left( \ln\left( \int_{0}^{r} Q(t)^{\frac{1}{2}} \ \mathrm{d}t \right) \right) \ \leq \ 
f_{k,m-1} \left( \ln\left( rQ(r)^{\frac{1}{2}} \right) \right) \ \leq \ f_{k,m-1} \left( \ln\left( Q(r) \right) \right)
$$
holds for every $r > R_{m}$ due to the monotonicity of $Q$ and $Q(r) > r^{2}$. Hence, we conclude to
$$
0 < \left( \int_{0}^{r} Q(t)^{\frac{1}{2}} \ \mathrm{d}t \right) 
f_{k,m-1} \left( \ln\left( \int_{0}^{r} Q(t)^{\frac{1}{2}} \ \mathrm{d}t \right) \right)
 \leq 
 r Q(r)^{\frac{1}{2}} f_{k,m-1} \left( \ln\left( Q(r) \right) \right)  <  Q(r)
$$
for $r \geq R_{m}$ by iii.) in Theorem \ref{Main_Thm_2}. Therefore there actually exists a gap between the radial
lower boundary function and the radial upper boundary function for $q$ to exist in for $|x| \geq R_{m}$.

\item By (\ref{Rosen_phi}) we have 
$$
-\ln \left( \varphi(x) \right) \ \leq \ \sqrt{2} \int_{0}^{|x|} Q(t)^{\frac{1}{2}} \ \mathrm{d}t + C
$$
for every $x \in \R^{n}$. We use the standard version of Young's inequality for increasing functions to include 
$\varepsilon > 0$. Define a function $f \colon \R \to [0,\infty)$ by
$$
f(q)=\left\{\begin{array}{ll} 
	f_{k,m-1}(q), & q \geq r_{0} \\
         f_{k,m-1}\left( r_{0} \right)\mathrm{e}^{\frac{q}{r_{0}}-1}, & q < r_{0}
\end{array}\right. 
$$
where $r_{0} = \ln \left( \int_{0}^{R_{m}} Q(t)^{\frac{1}{2}} \ \mathrm{d}t \right) > 0$. 
The composition of functions $f \circ \ln$ is continuous and strictly increasing on $(0,\infty)$ such that 
$f \left( \ln(t) \right) \to 0$ for $t \to 0^{+}$. Let $g$ be the inverse function of $f \circ \ln$  on $(0,\infty)$. 
Then for every $a,b > 0$ we infer that
$$
ab \ \leq \ \int_{0}^{a} f \left( \ln(t) \right) \ \mathrm{d}t + \int_{0}^{b} g(t)  \ \mathrm{d}t \ \leq \ a f \left( \ln(a) \right) + b g(b)
$$
is true. Let $x \in \R^{n}$ with $|x| \geq R_{m}$, $d \in (0,1)$ and $\varepsilon > 0$ be arbitrary but fixed. Define 
$a =  \int_{0}^{|x|} Q(t)^{\frac{1}{2}} \ \mathrm{d}t$ and $b = \frac{\sqrt{2}}{d \varepsilon}$ and conclude to
$$
\int_{0}^{|x|} Q(t)^{\frac{1}{2}} \ \mathrm{d}t \ \leq \ \frac{d \varepsilon}{\sqrt{2}} 
\left( \int_{0}^{|x|} Q(t)^{\frac{1}{2}} \ \mathrm{d}t \right)
f_{k,m-1} \left( \ln\left( \int_{0}^{|x|} Q(t)^{\frac{1}{2}} \ \mathrm{d}t \right) \right) + g\left( \frac{\sqrt{2}}{d \varepsilon} \right).
$$

\item Hence we conclude to
\begin{align*}
& -\ln \left( \varphi(x) \right) \ \leq \ \sqrt{2} \left( \int_{0}^{|x|} Q(t)^{\frac{1}{2}} \ \mathrm{d}t \right) + C \\
& \leq \ \varepsilon d \left( \int_{0}^{|x|} Q(t)^{\frac{1}{2}} \ \mathrm{d}t \right)
f_{k,m-1} \left( \ln\left( \int_{0}^{|x|} Q(t)^{\frac{1}{2}} \ \mathrm{d}t \right) \right)
 + \sqrt{2} \ g\left( \frac{\sqrt{2}}{d \varepsilon} \right) + C \\
& \leq \ \varepsilon q(x) + \underbrace{\sqrt{2} \ g\left( \frac{\sqrt{2}}{d \varepsilon} \right) + C}_{=\gamma(\varepsilon)} 
\end{align*}
for every $x \in \R^{n}$ with $|x| \geq R_{m}$. Again we argue as in (\ref{Rosen_phi}) that $C$ can be chosen sufficiently large such
that these Rosen inequalities hold for every $x \in \R^{n}$ which proves the claim.
\end{enumerate}

\subsection{Proof of Lemma \ref{Lemma_1}}

\begin{enumerate}[a)]

\item For $Q_{1}(r) = r^{\alpha}$ with $\alpha > 2$ the requirements of Theorem \ref{Main_Thm_2} are easy to verify since
$$
\left( \ln \left( Q(r) \right) \right)^{k} r Q(r)^{-\frac{1}{2}} \ = \ \frac{\alpha^{k} \left( \ln (r) \right)^{k}}{r^{\frac{\alpha}{2} - 1}} \to 0 
$$
and $Q^{\prime}(r)Q(r)^{-3/2} = \alpha r^{- ( \frac{\alpha}{2} + 1)} \to 0$ for $r \to \infty$.

\item Let $Q_{2}(r) = r^{2} \left( \ln(r) \right)^{\alpha}$ for $\alpha > 2$ and $r > \mathrm{e}$. Then 
$$
\ln \left( Q_{2}(r) \right) \ = \ 2 \ln(r) + \alpha \ln^{2} (r) \ \leq \ (2+\alpha) \ln(r) 
$$
holds for $r > \mathrm{e}$. For $k \in (1, \frac{\alpha}{2})$ we argue that
\begin{align*}
\frac{r \left( \ln \left( Q_{2}(r) \right) \right)^{k}}{Q_{2}(r)^{\frac{1}{2}}} \ \leq \ 
\left( 2 + \alpha \right)^{k} \frac{\left( \ln(r) \right)^{k}}{\left( \ln(r) \right)^{\frac{\alpha}{2}}} \ = \ 
\left( 2 + \alpha \right)^{k} \ln(r)^{k-\frac{\alpha}{2}} \ \to \ 0
\end{align*}
follows for $r \to \infty$. Hence, there is a radius $R>0$ such that
\begin{align*}
& \left( \int_{0}^{r} Q_{2}(t)^{\frac{1}{2}} \d t \right) \  \left( \ln \left(  \int_{0}^{r} Q_{2}(t)^{\frac{1}{2}} \d t  \right) \right)^{k}  
\leq r Q_{2}(r)^{\frac{1}{2}} \left( \ln \left( r Q_{2}(r)^{\frac{1}{2}} \right) \right)^{k} \\
& \leq r Q_{2}(r)^{\frac{1}{2}}  \left( \ln \left( Q_{2}(r) \right) \right)^{k} < Q_{2}(r)
\end{align*}
holds for every $r > R$ where we used the monotonicty of $Q_{2}$ and $r^{2} < Q_{2}(r)$. 

\item Let $Q_{3}(r) = r^{2} \left( \ln(r) \right)^{2} \left( \ln^{2}(r) \right)^{\alpha}$ for $\alpha > 2$ and $r > \mathrm{e}^{\mathrm{e}}$.
Then 
\begin{align*}
Q_{3}^{\prime}(r) & = 2 r \left( \ln(r) \right)^{2} \left( \ln^{2}(r) \right)^{\alpha} + 2 r \ln(r) \left( \ln^{2}(r) \right)^{\alpha} 
+ \alpha r \ln(r) \left( \ln^{2}(r) \right)^{\alpha-1} \\
& = 2 \frac{Q_{3}(r)}{r} + 2 \frac{Q_{3}(r)}{r\ln(r)} + \alpha \frac{Q_{3}(r)}{r\ln(r) \ln^{2}(r)}
\end{align*}
holds for every $r > \mathrm{e}^{\mathrm{e}}$. Therefore we easily see that $Q_{3}^{\prime}(r) Q_{3}(r)^{-\frac{3}{2}} \to 0$
follows for $r \to \infty$. Furthermore
\begin{enumerate}[1.)]

\item $\ln \left( Q_{3}(r) \right) = 2 \ln(r) + 2 \left( \ln^{2}(r) \right) + \alpha \left( \ln^{3}(r) \right) \leq (2+2+\alpha) \ln(r)$

\item $\ln^{2} \left( Q_{3}(r) \right) \leq \ln (2+2+\alpha) + \ln^{2}(r) \leq 2 \ln^{2}(r)$

\end{enumerate}
both hold for $r$ large. Hence, due to 1.) and 2.) we argue that
$$
\frac{r f_{k,1}\left( \ln \left( Q_{3}(r) \right) \right)}{Q_{3}(r)^{\frac{1}{2}}} \ \leq \ 2^{k} (2+2+\alpha) 
\frac{r \left( \ln^{2}(r) \right)^{k} \ln(r)}{r \ln(r) \left( \ln^{2}(r) \right)^{\frac{\alpha}{2}}} \ = \
2^{k} (2+2+\alpha) \left( \ln^{2}(r) \right)^{k-\frac{\alpha}{2}}
$$
holds for $r$ large. So, for $k \in (1,\frac{\alpha}{2})$ there exists a radius $R > 0$ such that
$$
0 \ < \ \frac{r f_{k,1}\left( \ln \left( Q_{3}(r) \right) \right)}{Q_{3}(r)^{\frac{1}{2}}} \ < \ 1
$$
holds for every $r \geq R$. All criteria of Theorem \ref{Main_Thm_2} are satisfied.

\item The general case $Q_{4}$ is shown by the same arguments as for $Q_{3}$.

\end{enumerate}

\section{Preparations to prove Intrinsic Ultracontractivity} \label{Definition_h_H}

\subsection{Properties of Schrödinger semigroups}

We still consider Schrödinger operators $H$ in $\mathrm{L}^{2}(\R^{n})$ as defined in Section \ref{definitions_and_preparations}.
Using Propostion 2.5 and Theorem 2.6 on page 50 in \cite{Ouhabaz05} we state the following properties.\\

\begin{lem} \label{positiv_H}
\
\begin{enumerate}[i.)]
\item The set $\mathrm{L}^{2}(\R^{n};\R)$ is invariant under 
$\mathrm{e}^{-tH}$ for $t>0$. We call $\mathrm{e}^{-tH}$ real.
\item The set $\mathrm{L}^{2}(\R^{n};[0,\infty))$ is invariant under $\mathrm{e}^{-tH}$ for $t>0$.
We call $\mathrm{e}^{-tH}$ positive.\\
\end{enumerate}
\end{lem} 

Furthermore, for every time $t > 0$ the function $\mathrm{e}^{-tH}u$ is even strictly positive 
almost everywhere in $\R^{n}$ for $u \in \mathrm{L}^{2}(\R^{n})$ being non-negative almost everywhere in $\R^{n}$. 
Hence, we call the operators $\mathrm{e}^{-tH}$ \textit{positivity improving}. The proof of the following theorem 
is found in appendix \ref{e^{-tH}_positivity_improving}.\\

\begin{thm}
For every $t>0$ the operators $\mathrm{e}^{-tH}$ are positivity improving.\\
\end{thm}

\subsection{Weighted $L^{p}$-spaces and weighted Schrödinger Semigroups}

We follow B. Davies in Section 4.2 on page 111 in \cite{Davies07} by introducing weighted Lebesgue 
space $\mathrm{L}^{p}_{\mu}(\R^{n})$ for $p \in [1,\infty)$ and weighted Schrödinger semigroups 
on $\mathrm{L}^{2}_{\mu}(\R^{n})$. For any Borel set $B$ in $\R^{n}$ we define a probability measure 
$\mu$ on $\R^{n}$ by 
$$
\mu (B) = \int_{B} \varphi (x)^{2} \d x \in [0,1]
$$
and for every $p \in [1,\infty)$ we denote \textit{weighted Lebesgue function spaces} $\mathrm{L}_{\mu}^{p}(\R^{n})$ 
as the set of measurable functions $u \colon \R^{n} \to \C$ that satisfy
$$
\int_{\R^{n}} |u(x)|^{p} \d \mu(x) = 
\int_{\R^{n}} |u(x)|^{p} \ \varphi (x)^{2} \d x < \infty.
$$ 
Furthermore, $\mathrm{L}_{\mu}^{\infty}(\R^{n})$ is defined as $\mathrm{L}^{\infty}(\R^{n})$ which is justified since 
Lebesgue null sets with respect to $\mu$ coincide with null sets to the regular Lebesgue measure on $\R^{n}$ 
as $\varphi$ is strictly positive in $\R^{n}$. Setting
$$
\| u \|_{p, \mu} = \left( \int_{\R^{n}} |u(x)|^{p} \ \mathrm{d}\mu(x) \right)^{\frac{1}{p}}
$$
for $u \in \mathrm{L}_{\mu}^{p}(\R^{n})$ turns $\mathrm{L}_{\mu}^{p}(\R^{n})$ into Banach spaces. For 
$u \in \mathrm{L}^{2}_{\mu}(\R^{n})$ we define a $C_{0}$-semigroup of contractions in $\mathrm{L}^{2}_{\mu}(\R^{n})$ by
\begin{equation}
\mathrm{e}^{-t\tilde{H}}u (x)  = \frac{1}{\varphi(x)}\mathrm{e}^{-tH}(\varphi u) (x)
\end{equation}
which we call the \textit{weighted Schrödinger semigroup} of $H$. The generator is given by 
\begin{equation}
-\tilde{H}u \ = \ - \frac{1}{\varphi} H (\varphi u) 
\end{equation} 
for $u \in D( \tilde{H}) = \varphi^{-1} D\left( H \right) \subseteq \mathrm{L}^{2}_{\mu}(\R^{n})$ where we use a sloppy
notation by writing $\varphi^{-1} D\left( H \right)$ for the set of products of $\varphi^{-1}$ with every function of 
the domain $D\left( H \right)$.\\

\begin{thm}\label{Contraction_weightedSemigroup} \ 
\begin{enumerate}[i.)]
\item For every $t > 0$ the operator $\mathrm{e}^{-t\tilde{H}}$ is a contraction in $\mathrm{L}_{\mu}^{1}(\R^{n}, \R)$.
\item For every $t > 0$ the operator $\mathrm{e}^{-t\tilde{H}}$ is a contraction in $\mathrm{L}_{\mu}^{\infty}(\R^{n}, \R)$.
\item For every $t>0$ and $p \in (1,\infty)$ the operator $\mathrm{e}^{-t\tilde{H}}$ is a contraction in 
$\mathrm{L}_{\mu}^{p}(\R^{n}, \R)$.
\end{enumerate}
\end{thm}

\begin{proof} \
\begin{enumerate}[i.)]
\item Let $t > 0$ be arbitrary but fixed.
\begin{enumerate}[a)]
\item Let $u \in \mathrm{L}_{\mu}^{2}(\R^{n}, \R) \subseteq \mathrm{L}_{\mu}^{1}(\R^{n}, \R)$ be non-negative almost everywhere 
in $\R^{n}$. So we reason that
\begin{align*}
& \int_{\R^{n}} |\mathrm{e}^{-t\tilde{H}}u  \ (x)| \d \mu (x) 
= \int_{\R^{n}} \mathrm{e}^{-t\tilde{H}}u \ (x) \d \mu (x) = \left\langle  \mathrm{e}^{-tH} \varphi u, \varphi \right\rangle 
= \left\langle \varphi u,  \mathrm{e}^{-tH} \varphi \right\rangle \\
& \\
& =  \left\langle \varphi u,   \mathrm{e}^{-tE_{0}} \varphi \right\rangle 
=  \mathrm{e}^{-tE_{0}} \left\langle \varphi u,  \varphi \right\rangle \leq \| u \|_{1,\mu} 
\end{align*}
is true where we used $\mathrm{e}^{-tH} \varphi = \mathrm{e}^{-tE_{0}} \varphi$ and $E_{0} \geq 0$.
\item For $u \in \mathrm{L}_{\mu}^{2}(\R^{n}, \R)$ we denote the positive and negative part of $u$
by $u^{+}$ and $u^{-}$ respectively with $u = u^{+} - u^{-}$. Using a) we end up with
\begin{align*}
& \int_{\R^{n}} |\mathrm{e}^{-t\tilde{H}}u \ (x)| \d \mu (x) 
\leq \int_{\R^{n}} |\mathrm{e}^{-t\tilde{H}}u^{+} \ (x)| \d \mu (x) 
+ \int_{\R^{n}} |\mathrm{e}^{-t\tilde{H}}u^{-} \ (x)| \d \mu (x) \\
& \leq \int_{\R^{n}} u^{+} (x) \d \mu (x) + \int_{\R^{n}} u^{-} (x) \d \mu (x) 
 = \int_{\R^{n}} |u (x)| \d \mu (x) = \| u \|_{1,\mu}.
\end{align*}
\item Now let $u \in \mathrm{L}_{\mu}^{1}(\R^{n}, \R)$. We define a sequence
$u_{k} = 1_{\{ u \leq k \}} \ u$ for $k \in \N$. Mind that $u_{k}$ is a bounded function on $\R^{n}$ and therefore is contained in 
$\mathrm{L}_{\mu}^{2}(\R^{n}, \R)$ since $\mu$ is a probability measure. Furthermore $(u_{k})$ converges to $u$ in 
$\mathrm{L}_{\mu}^{1}(\R^{n})$ by Lebesgue's theorem 
since $u_{k}$ converges pointwise almost everywhere to $u$ in $\R^{n}$ and 
$|u_{k} - u| \leq 2|u| \in \mathrm{L}_{\mu}^{1}(\R^{n})$ is true as well. So using b) we conclude that 
$(\mathrm{e}^{-t\tilde{H}}u_{k})$ is a Cauchy sequence in $\mathrm{L}_{\mu}^{1}(\R^{n})$. Hence there 
exists a unique and bounded extension of $\mathrm{e}^{-t\tilde{H}}$ to $\mathrm{L}_{\mu}^{1}(\R^{n})$ that 
proves the claim.
\end{enumerate}
\item Let $t > 0$ be chosen arbitrary but fixed and 
$u \in \mathrm{L}_{\mu}^{\infty}(\R^{n}, \R) \subseteq \mathrm{L}_{\mu}^{2}(\R^{n}, \R)$ be non-negative 
almost everywhere in $\R^{n}$. Using that $\mathrm{e}^{-tH}$ is positivity improving we argue that
\begin{align*}
0 & < \mathrm{e}^{-tH} \left( \varphi \| u \|_{\infty} - \varphi u \right) (x) = 
\mathrm{e}^{-tH} \left( \varphi \| u \|_{\infty}\right) (x) - \mathrm{e}^{-tH} \left( \varphi u \right) (x) \\
\ \\
& =
\| u \|_{\infty} \mathrm{e}^{-tE_{0}} \varphi(x) - \mathrm{e}^{-tH} \left( \varphi u \right) (x) 
\end{align*}
holds for almost every $x \in \R^{n}$. Hence,
$$
0 < \mathrm{e}^{-tH} \left( \varphi u \right) (x) < \| u \|_{\infty} \mathrm{e}^{-tE_{0}} \varphi(x) = 
\frac{\| u \|_{\infty}}{\mathrm{e}^{tE_{0}}} \ \varphi(x)
$$
is implied for almost every $x \in \R^{n}$ which gives $\| \mathrm{e}^{-t\tilde{H}} u \|_{\infty} < \mathrm{e}^{-tE_{0}} \| u \|_{\infty}$
for the weighted Schrödinger semigroup $\mathrm{e}^{-t\tilde{H}}$. For $u \in \mathrm{L}_{\mu}^{\infty}(\R^{n}, \R)$ not necesserarily
non-negative we write $u = u^{+} - u^{-}$ and argue as in i.) to prove the claim.
\item Using i.) and ii.) the claim is implied by the Riesz-Thorin interpolation theorem.
\end{enumerate}
\end{proof}

\subsection{Rosen's lemma} \label{Rosen_lemma}

Let the ground state $\varphi$ of $H$ satisfy Rosen inequalities, i.e.
\begin{equation} \label{Rosen_inequalities}
-\ln \left( \varphi (x) \right) \ \leq \ \varepsilon q(x) + \gamma(\varepsilon)
\end{equation} 
for every $\varepsilon > 0$ and an associated $\gamma(\varepsilon) > 0$. Furthermore let
$u \in D(\tilde{H}) \cap \mathrm{L}^{\infty}_{\mu}(\R^{n})$ be non-negative almost everywhere in $\R^{n}$.
Then for every $\varepsilon > 0$ the Logarithmic Sobolev inequalities of $\tilde{H}$ 
\begin{align*}
& \int_{\R^{n}} \left( \mathrm{e}^{-t\tilde{H}} u(x) \right)^{p} \ \ln \left( \mathrm{e}^{-t\tilde{H}} u (x) \right) \d \mu(x) \leq \\ 
& \varepsilon \ 
\langle \tilde{H} \mathrm{e}^{-t\tilde{H}} u, (\ \mathrm{e}^{-t\tilde{H}} u \ )^{p-1} \rangle_{\mu} 
 + \frac{2\beta(\varepsilon)}{p} \ \| \mathrm{e}^{-t\tilde{H}} u \|_{p, \mu}^{p} + \| \mathrm{e}^{-t\tilde{H}} u \|_{p, \mu}^{p}  \ln  \| \mathrm{e}^{-t\tilde{H}} u \|_{p, \mu}
\end{align*}
are implied with $\beta(\varepsilon) = \frac{\varepsilon}{2} - \frac{n}{4} \ln( \frac{\varepsilon}{2}) + \gamma( \frac{\varepsilon}{2}) + C$
and a constant $C \in \R$. Please compare with Corollary 4.4.2 on page 118 first and then with Lemma 2.2.6 on page
67 in \cite{Davies07}.\\

\section{Intrinsic ultracontractivity of $\mathrm{e}^{-tH}$} \label{IU_e^-tH}

In the following we refer to potentials $q \colon \R^{n} \to (0,\infty)$ of Theorem \ref{Main_Thm_2} such that
there exists a radius $R_{m} > 0$, $k>1$, $d \in (0,1]$ and $m \in \N$ with 
\begin{equation}
d \left( \int_{0}^{ |x|} Q(t)^{\frac{1}{2}} \ \mathrm{d}t \right) \
f_{k, m-1} \left( \ln \left( \int_{0}^{ |x|} Q(t)^{\frac{1}{2}} \ \mathrm{d}t \right) \right) \ \leq q(x) \ \leq \ Q \left( |x| \right)
\end{equation}
for every $|x| \geq R_{m}$ where $f_{k, m-1}$ is given by
\begin{equation}
f_{k, m-1} (t) = \left( \ln^{(m-1)}(t) \right)^{k} \ \prod_{p=0}^{m-2} \ln^{(p)}(t).
\end{equation} 
Define $f \colon \R \to (0,\infty)$ by
$$
f(q)=\left\{\begin{array}{ll} 
	f_{k,m-1}(q), & q \geq r_{0} \\
         f_{k,m-1}\left( r_{0} \right)\mathrm{e}^{\frac{q}{r_{0}}-1}, & q < r_{0}
\end{array}\right. 
$$
with $r_{0} = \ln \left( \int_{0}^{R_{m}} Q(t)^{\frac{1}{2}} \ \mathrm{d}t \right)$ and $g$ as the inverse of 
$f \circ \ln$ in $(0,\infty)$. Then the ground state 
$\varphi$ satisfies Rosen inequalities with
\begin{equation} \label{gamma}
\gamma(\varepsilon) = \sqrt{2} \ g\left( \frac{\sqrt{2}}{d \cdot \varepsilon} \right) + C
\end{equation}
for a constant $C > 0$. For details please return to Section \ref{Proofs_for_Rosen_Inequalities}.\\

\begin{thm}\label{IU_of_e^{-tH}}
Let $q$ be as described above then the Schrödinger semigroup $\mathrm{e}^{-tH}$ is 
intrinsic ultracontractive in $\mathrm{L}^{2}(\R^{n})$, i.e.
\begin{equation}
\forall t > 0 \exists C_{t} > 0 \ : \  \bigl| \mathrm{e}^{-tH}u (x) \bigr| \ \leq \ C_{t} \|u\|_{2} \ \varphi (x) \\
\end{equation}
holds almost everywhere in $\R^{n}$ for every $u \in \mathrm{L}^{2}(\R^{n})$.\\
\end{thm}

\section{Proof of Theorem \ref{IU_of_e^{-tH}}} \label{Proof_IU}

\subsection{Preliminary} \label{Preliminary_Proof_IU}

We define $T = \frac{\sqrt{2}}{d} \int_{\frac{1}{2}\ln(2)}^{\infty} f(r)^{-1} \d r \ \in (0,\infty)$. Please compare with
Appendix \ref{Miscellaneous} for details on the integration. Let $t \in (0,T]$ be fixed. Then
there exists a $\xi (t) \geq 0$ with
$$
 \frac{\sqrt{2}}{d} \int_{\frac{1}{2}\ln(2) + \xi (t)}^{\infty} f(r)^{-1} \d r \ = \ t
$$
by the intermediate value theorem. Please mind that $\xi$ is a decreasing function with $\xi (t) \to \infty$ for $t \to 0$.
We define a function

\begin{equation}
\varepsilon_{t}(p) = \frac{1}{\sqrt{2} d \ f( \frac{1}{2}\ln(p) + \xi (t))}
\end{equation}
for $p \in [2,\infty)$. Then
\begin{equation} \label{limit_G}
\int_{2}^{\infty} \frac{\varepsilon_{t}(p)}{p} \d p = 
\frac{\sqrt{2}}{d} \int_{\frac{1}{2}\ln(2) + \xi (t)}^{\infty} f(r)^{-1} \d r = t
\end{equation}
is implied. We define $p \colon [0,t) \to [2,\infty)$ as a solution of the initial value problem
$$
\left\{
\begin{array}{ll} 
p^{\prime}(s) & = \frac{p(s)}{\varepsilon_{t}( p(s))} \ \text{for } s \in (0,t) \\
p(0) & = 2
\end{array}\right. 
$$
by separation of variables. Therefore, $G \colon [2,\infty) \to \R$ is defined by
$$
G(r) = \int_{2}^{r} \frac{\varepsilon_{t}(s)}{s} \d s.
$$
Then $G^{\prime}(r) = \frac{\varepsilon_{t}(r)}{r} > 0$ is true for every $r \geq 2$ and so $G$ is strictly monotone increasing
and continuous on $[2,\infty)$ with $G(2)=0$. Using (\ref{limit_G}) we state $G(r) \to t$ for $r \to \infty$. 
Hence $G$ is a bijection from $[2,\infty)$ to $[0,t)$ and we define $p$ as the inverse function $G^{-1}$ of $G$ on $[0,t)$. 
Hence $p$ satisfies $p(0)=2$ and
\begin{equation}
p^{\prime}(s) = \frac{1}{G^{\prime}\bigl( p(s) \bigr)} = \frac{p(s)}{\varepsilon_{t}(p(s))}
\end{equation}
for every $s \in [0,t)$ by the inverse function theorem. Finally, we define a function $N$ by
\begin{equation}
N(s) \ = \ 2 \int_{0}^{s} \frac{\beta\left( \varepsilon_{t}\left( p(r) \right) \right)}{\varepsilon_{t}\left( p(r) \right) p(r)} \d r
\ = \ 2\int_{2}^{p(s)} \frac{\beta\left( \varepsilon_{t}(q) \right)}{q^{2}} \d q
\end{equation}
for $s \in [0,t)$. Using Appendix \ref{Miscellaneous} we define the limit of $N$ for $s \to t$ by $M(t) = \lim_{s \to t} N(s)$
since $p(s) \to \infty$ for $s \to t$.\\

\subsection{The core argument}

Let $u \in D(\tilde{H}) \cap \mathrm{L}_{\mu}^{\infty}(\R^{n})$ be non-negative almost everywhere in $\R^{n}$. 
Due to the definition of weighted Schrödinger semigroups, 
$\mathrm{e}^{-t\tilde{H}}u \in \mathrm{L}_{\mu}^{\infty}(\R^{n}) \cap \mathrm{L}_{\mu}^{p}(\R^{n})$ is implied
for every $p \in [2,\infty)$ since $\mu$ is a probability measure. Furthermore,
\begin{equation}
\|\mathrm{e}^{-t\tilde{H}}u \|_{\infty, \mu}  = \lim_{s \to t} \|\mathrm{e}^{-t\tilde{H}}u \|_{p(s), \mu} 
\end{equation}
is true and
\begin{equation} \label{Abschätzung_t-s}
\|\mathrm{e}^{-t\tilde{H}}u \|_{p(s), \mu} \ \leq \ \|\mathrm{e}^{-(t-s)\tilde{H}} \|_{p(s) \to p(s), \mu} \
\|\mathrm{e}^{-s\tilde{H}}u \|_{p(s), \mu}  \ \leq \  \|\mathrm{e}^{-s\tilde{H}}u \|_{p(s), \mu}
\end{equation}
holds for every $s \in (0,t)$ since $\mathrm{e}^{-(t-s)\tilde{H}}$ is a contraction in $\mathrm{L}_{\mu}^{p(s)}(\R^{n}, \R)$
by Theorem \ref{Contraction_weightedSemigroup}. We focus on proving that
$$
\left\{ s \mapsto \|\mathrm{e}^{-s\tilde{H}}u \|_{p(s), \mu} \right\}
$$
is monotonically decreasing on $[0,t)$ where
$\|\mathrm{e}^{-0\tilde{H}}u \|_{p(0), \mu} =  \| u \|_{2, \mu} = \| \varphi u \|_{2}$ is true.
Following Lemma 2.2.2 on page 64 in \cite{Davies07} we imply
\begin{align*}
& \frac{\d}{\d s} \|\mathrm{e}^{-s\tilde{H}}u \|^{p(s)}_{p(s), \mu} \\
& = \ p(s) \langle -\tilde{H}\mathrm{e}^{-s\tilde{H}}u, \left( \mathrm{e}^{-s\tilde{H}}u \right)^{p(s)-1}  \rangle_{\mu} 
+ p^{\prime}(s) \int_{\R^{n}} \left( \mathrm{e}^{-s\tilde{H}}u (x) \right)^{p(s)} \ln \left( \mathrm{e}^{-s\tilde{H}}u (x) \right) \d \mu(x)
\end{align*}
is true. Let us include the function $N$ to our argumentation to use the Logarithmic Sobolev inequalities of $\tilde{H}$ mentioned 
in Section \ref{Rosen_lemma}. 
Mind that the derivative of
$$
\ln \left( \ \mathrm{e}^{-N(s)}  \| \mathrm{e}^{-s\tilde{H}} u \|_{p(s),\mu} \ \right) 
= - N(s) + \frac{1}{p(s)} \ln \| \mathrm{e}^{-s\tilde{H}} u \|_{p(s),\mu}^{p(s)}
$$
is equal to
\begin{equation*}
 - \underbrace{\frac{2\beta\left( \varepsilon_{t}\left( p(s) \right) \right)}{\varepsilon_{t}\left( p(s) \right) p(s)}}_{=N^{\prime}(s)} 
\underbrace{- \frac{1}{p(s)  \varepsilon_{t}\left( p(s) \right)}}_{=\frac{\d}{\d s} \frac{1}{p(s)}} \  
\ln \left\| \mathrm{e}^{-s\tilde{H}} u \right\|_{p(s),\mu}^{p(s)}
 + \frac{1}{p(s)} \ \left\| \mathrm{e}^{-s\tilde{H}} u \right\|_{p(s),\mu}^{-p(s)} \ 
\frac{\d}{\d s} \left\| \mathrm{e}^{-s\tilde{H}} u \right\|_{p(s),\mu}^{p(s)}.
\end{equation*}
Hence, we end up with
\begin{align*}
& \frac{\d}{\d s} \ \ln \left( \ \mathrm{e}^{-N(s)}  \| \mathrm{e}^{-s\tilde{H}} u \|_{p(s),\mu} \ \right)  = \\
\ \\
& \frac{1}{ \varepsilon_{t}\left( p(s) \right) } \ \left\| \mathrm{e}^{-s\tilde{H}} u \right\|_{p(s),\mu}^{-p(s)}
 \left\{ \ \int_{\R^{n}} (\mathrm{e}^{-s\tilde{H}} u (x) )^{p(s)} \ \ln \left( \mathrm{e}^{-s\tilde{H}} u(x) \right) \d \mu(x) \right. \\
& - \varepsilon_{t}\left( p(s) \right) \  
\left\langle \tilde{H} \mathrm{e}^{-s\tilde{H}} u, (\mathrm{e}^{-s\tilde{H}} u)^{p(s)-1} \right\rangle_{\mu}
 - \frac{2\beta\left( \varepsilon_{t}\left( p(s) \right) \right)}{p(s)} \ \left\| \mathrm{e}^{-s\tilde{H}} u \right\|_{p(s),\mu}^{p(s)} \\
& \left.
 - \left\| \mathrm{e}^{-s\tilde{H}} u \right\|_{p(s),\mu}^{p(s)} \ 
\ln \left\| \mathrm{e}^{-s\tilde{H}} u \right\|_{p(s),\mu} \right\} \leq 0
\end{align*}
for every $s \in [0,t)$ due to the Logarithmic Sobolev inequalities of $\tilde{H}$ in Rosen's lemma \ref{Rosen_lemma}. 
Therefore,
\begin{equation}\label{Abschätzung_p(s)-2}
 \mathrm{e}^{-N(s)}   \left\| \mathrm{e}^{-s\tilde{H}} u \right\|_{p(s),\mu}
 \ \leq \  \mathrm{e}^{-N(0)}  \| u \|_{p(0),\mu} = \| u \|_{2,\mu}
\end{equation}
is implied for $s \in [0,t)$. So, for every $t \in (0,T]$ we infer
\begin{equation}\label{finale_Abschätzung}
\| \mathrm{e}^{-t\tilde{H}}u \|_{\infty,\mu} \ 
= \ \lim_{s \to t}  \| \mathrm{e}^{-t\tilde{H}}u \|_{p(s),\mu}
 \leq \ 
\lim_{s \to t} \mathrm{e}^{N(s)} \ \| u \|_{2,\mu} =  \mathrm{e}^{M(t)} \ \| u \|_{2,\mu}
\end{equation}
is true due to (\ref{Abschätzung_t-s}) and (\ref{Abschätzung_p(s)-2}).
For $t > T$ there exists a $k \in \N$ such that $kT < t \leq (k+1)T$ which implies $0 < t -kT \leq T$. 
Therefore we conclude to 
$$
\left|  \mathrm{e}^{-t\tilde{H}}u (x) \right| = \left|  \mathrm{e}^{-(t-kT)\tilde{H}} \left( \mathrm{e}^{-kT\tilde{H}}u \right) (x) \right|
\leq  \mathrm{e}^{M(t-kT)} \ 
\|  \mathrm{e}^{-kT\tilde{H}} u \|_{2,\mu}
\leq \mathrm{e}^{M(t-kT)} \ \| u \|_{2,\mu}
$$
almost everywhere in $\R^{n}$ by (\ref{finale_Abschätzung}). So, for every $t > 0$ there exists
a constant $C_{t} > 0$ with
$
\| \mathrm{e}^{-t\tilde{H}}u \|_{\infty,\mu}  \leq \ C_{t} \  \| u \|_{2,\mu}
$
which implies
\begin{equation}\label{IU_Abschätzung}
\left| \left( \mathrm{e}^{-tH} \varphi u \right) (x) \right| \ \leq \ C_{t} \ \varphi (x) \ \| \varphi u\|_{2}  
\end{equation}
almost everywhere in $\R^{n}$ for $u \in D(\tilde{H}) \cap \mathrm{L}^{\infty}_{\mu}(\mathbb{R}^{n})$ 
being non-negative almost everywhere in $\mathbb{R}^{n}$.

\subsection{The final part}

Using $D(\tilde{H}) = \varphi^{-1} D(H)$ we rewrite (\ref{IU_Abschätzung}) to 
\begin{equation}\label{IU_inequality}
\left| \left( \mathrm{e}^{-tH} v \right) (x) \right| \ \leq \ C_{t} \varphi(x) \| v \|_{2}
\end{equation}
for every $v \in D(H)$ being non-negative such that $\frac{1}{\varphi} v \in \mathrm{L}^{\infty}(\mathbb{R}^{n})$. 
By Lemma \ref{ground_state_continuous} and Lemma \ref{strict_positivity_ground_state} 
we state that $\varphi^{-1}$ is bounded on every compact
set $M$ in $\R^{n}$. So, for every non-negative function $v$ in $C_{c}^{\infty}(\R^{n}) \subseteq D(H)$ 
we conclude that  $\frac{1}{\varphi} v$ is bounded. Hence (\ref{IU_inequality})
holds for every non-negative $v \in C_{c}^{\infty}(\R^{n})$ and therefore holds for every non-negative function
in $\mathrm{L}^{2}(\R^{n})$ by density argumentation.\\

Now, let $u \in \mathrm{L}^{2}(\R^{n})$ be real-valued but not necessarily non-negative almost everywhere in $\R^{n}$. 
We use the positive and negative parts $u^{+}$ and $u^{-}$ of $u$ which are contained  in $\mathrm{L}^{2}(\R^{n})$
and non-negative almost everywhere in $\R^{n}$. By Lemma \ref{positiv_H} the operator $\mathrm{e}^{-tH}$ is
not only linear but also positive. So, we conclude that
$\left( \mathrm{e}^{-tH}u \right)^{+} = \mathrm{e}^{-tH}u^{+}$
is true which means that the positive part of $ \mathrm{e}^{-tH}u$ is given by $\mathrm{e}^{-tH}u^{+}$. 
A similar argument can be made for the negative part of $ \mathrm{e}^{-tH}u$. Therefore (\ref{IU_inequality}) holds 
also in the case of any real-valued function $u \in \mathrm{L}^{2}(\R^{n})$.\\

Finally, let $u \in \mathrm{L}^{2}(\R^{n})$ but not necessarily real-valued. Then, the argumentation is similar  
since the real and imaginary parts of $u$ are real-valued functions in $\mathrm{L}^{2}(\R^{n})$ and 
$\mathrm{e}^{-tH}$ is linear and real again by Lemma \ref{positiv_H}. We conclude that 
(\ref{IU_inequality}) holds for a general function $u \in \mathrm{L}^{2}(\R^{n})$ which proves Theorem \ref{IU_of_e^{-tH}}.

\newpage

\appendix

\section{Agmon's version of the comparison principle}

In contrast to \cite{AlziaryTakac09} we use Agmon's version of the comparison principle in \cite{Hundertmark_Jex_Lange_2023}. 
Let $h$ be the quadratic form of Section \ref{definitions_and_preparations} 
with the associated Schrödinger operator $H$ on $\mathrm{L}^{2} \left( \R^{n} \right)$.\\

\begin{dfn} \label{Definition_sub_supersolution}
Let $U \subseteq \R^{n}$ be an open set.
\begin{enumerate}[a)]
\item The local quadratic form domain $D_{loc}^{U}(h)$ is given by
\begin{equation}
D_{loc}^{U}(h) = \left\{ u \in \mathrm{L}^{2}_{loc} \left( U \right) \ : \ \chi u \in D(h) \text{ for every } 
\chi \in C_{c}^{\infty}\left( U \right) \right\}.
\end{equation}
It is the vector space of functions that are locally in $D(h)$.
\item $u$ is a supersolution of $H$ in $U$ with energy $E$, if $u \in D_{loc}^{U}(h)$ and
\begin{equation}
h\left( u, \xi \right) - E \langle u, \xi  \rangle \ \geq \ 0 
\end{equation}
are satisfied for every non-negative $\xi \in C_{c}^{\infty}\left( U \right)$. Please note that we have slightly abused 
notation here since $u$ might not be contained in $D(h)$ but we understand  
$h\left( u, \xi \right)$ as 
$\langle \nabla u, \nabla \xi \rangle + \langle q(x)u, \xi \rangle$.
\item $u$ is a subsolution of $H$ in $U$ with energy $E$, if $u \in D_{loc}^{U}(h)$ and
\begin{equation}
h\left( u, \xi \right) - E \langle u, \xi  \rangle \ \leq \ 0 
\end{equation}
are satisfied for every non-negative $\xi \in C_{c}^{\infty}\left( U \right)$.\\
\end{enumerate}
\end{dfn}

Theorem \ref{comparision_principle} is taken from \cite{Hundertmark_Jex_Lange_2023} in form of
Theorem 2.7 on page 11. \\

\begin{thm}[Agmon's version of the comparison principle] \label{comparision_principle} \ \\
Let $R > 0$ and $u$ be a positive supersolution of $H$ with energy $E$ in a neighborhood of infinity 
$\Omega_{R} = \left\{ x \in \R^{n} \ : \ |x| > R \right\}$. Furthermore, let $v$ be a subsolution of $H$ in 
$\Omega_{R}$ with energy $E$ such that
\begin{equation}\label{L^2_property_Agmon}
\liminf_{N \to \infty} \left( \frac{1}{N^{2}} \int_{N \leq |x| \leq \alpha N} \left| v \right|^{2} \d x \right) = 0
\end{equation}
holds for some $\alpha > 1$. If for some $\delta > 0$ and $0 \leq C < \infty$ one has
\begin{equation}\label{boundary_property_Agmon}
v(x) \ \leq \ Cu(x) 
\end{equation} 
on the annulus $R < |x| \leq R + \delta$, then $v(x) \ \leq \ Cu(x)$ holds for almost every $x \in \Omega_{R}$.\\
\end{thm}  

\begin{rem} \label{Remark_comparision_principle} \ 
\begin{enumerate}[i.)]
\item In its original form  \cite{Agmon85}
Agmon assumes that the supersolution $u$ as well as the subsolution $v$ are both continuous in $\overline{\Omega_{R}}$. 
However, this additional assumption of $u$ and $v$ being continuous are only made to guarantee that
condition (\ref{boundary_property_Agmon}) holds for a constant $C = c_{2}/c_{1}$ with 
$c_{1} = \inf_{R \leq |x| \leq R+ \delta} u(x)$ and $c_{2} = \sup_{R \leq |x| \leq R+ \delta} \left| v(x) \right|$ and arbitrary $\delta > 0$.
Please see Remark 2.9 on page 12 in \cite{Hundertmark_Jex_Lange_2023} for further details.
\item Furthermore, for $v \in \mathrm{L}^{2}\left( \Omega_{R} \right)$ condition (\ref{L^2_property_Agmon}) is already implied. \\
\end{enumerate}
\end{rem}

\section{Positivity of Schrödinger semigroups}\label{e^{-tH}_positivity_improving}

To use Lemma \ref{CharacterizationIrreducibilty} we first have to show that $h$ satisfies the 
demanded property.\\

\begin{prp}
Let $u, v \in D(h)$ have disjoint supports. Then $h(u,v) = 0$ follows.
\end{prp}

\begin{proof}
Let $u, v \in D(h)$ with $\mathrm{supp} (u) \cap \mathrm{supp} (v) = \emptyset$.\\

First, let us deal with the special cases that the support of $u$ is either $\R^{n}$ or $\emptyset$.
If $\mathrm{supp} (u) = \R^{n}$ then $\mathrm{supp} (v)$ has to be empty. That means $v$ is $0$ almost everywhere
in $\R^{n}$ which implies $h(u,v) = 0$. A similar argument holds 
for the case of $\mathrm{supp} (u) = \emptyset$. Then $u$ is the function being $0$ almost everywhere in $\R^{n}$.
Once again $h(u,v) = 0$ follows in this case.\\

Now, let neither $\mathrm{supp} (u)$ nor $\mathrm{supp} (v)$ be $\emptyset$ or $\R^{n}$. We define 
the closed sets $S_{1}$ as $\mathrm{supp} (u)$ and $S_{2}$ as $\mathrm{supp} (v)$. Remember that 
$S_{1} \cap S_{2} = \emptyset$. The $j$-th weak derivative of $u$ is equal to $0$ on $\R^{n} \setminus S_{1}$.
For $\psi \in C_{c}^{\infty}(\R^{n})$ such that $\mathrm{supp}(\psi) \subset \R^{n} \setminus S_{1}$ we get
$$
0 = - \int_{\R^{n}} u(x) \ \overline{\partial_{j} \psi (x)} \ \mathrm{d}x = \int_{\R^{n}} \partial_{j} u (x) \  \overline{\psi (x)} \ \mathrm{d}x = 
\int_{\R^{n} \setminus S_{1}} \partial_{j} u (x) \  \overline{\psi (x)} \ \mathrm{d}x.
$$
The fundamental lemma in calculus of variations states that $\partial_{j} u$ is $0$ on $\R^{n} \setminus S_{1}$.
A similar statement is true on $\R^{n} \setminus S_{2}$ for the $j$-th weak derivative of $v$. Hence 
$\partial_{j} u = 1_{S_{1}} \partial_{j} u$ and  $\partial_{j} v = 1_{S_{2}} \partial_{j} v$ holds. We conclude 
that
$$
h(u,v) = \bigl\{ \sum_{j=1}^{n} \langle \partial_{j} u,  \partial_{j} v \rangle   \bigr\} 
+ \langle q^{\frac{1}{2}} u, q^{\frac{1}{2}} v \rangle
 =  \sum_{j=1}^{n} \langle 1_{S_{1}} \partial_{j} u, 1_{S_{2}} \partial_{j} v \rangle = 0
$$
which proves the claim.\\
\end{proof}

We present the following characterization from \cite{Ouhabaz05} as Corollary 2.11 on page 54.\\

\begin{lem}\label{CharacterizationIrreducibilty}
If $h$ satisfies $h(u, v) = 0$ for functions $u, v \in D(h)$ having disjoint 
supports then the following statements are equivalent:
 \begin{enumerate}[i.)]
\item For every $t>0$ the operators $\mathrm{e}^{-tH}$ are positivity improving.
\item If $\Omega \subseteq \R^{n}$ satisfies $1_{\Omega} \ u \in D(h)$ for every 
$u \in D(h)$ then either $\Omega$ or $\R^{n} \setminus \Omega$ is a Lebesgue null set.\\
\end{enumerate} 
\end{lem} 

We use this Lemma \ref{CharacterizationIrreducibilty} to prove that $\mathrm{e}^{-tH}$ is positivity improving for every $t > 0$.\\

\begin{thm}
For every $t>0$ the operators $\mathrm{e}^{-tH}$ is positivity improving.\\
\end{thm}

\begin{proof}
We prove by contradiction and assume that $\{ \mathrm{e}^{-tH} \ | \  t \geq 0 \}$ is not positivity improving. 
By Lemma \ref{CharacterizationIrreducibilty} there is a set $\Omega \subseteq \R^{n}$ such that  
$1_{\Omega} u$  is contained in $D(h)$ for every $u \in D(h)$ but neither $\Omega$ nor 
$\Omega^{C}  = \R^{n} \setminus \Omega$ is a Lebesgue null set. We denote the Lebesgue measure in 
$\R^{n}$ by $\lambda$.\\

\begin{enumerate}[i.)]
\item We show that there exists a $x_{0} \in \R^{n}$ such that neither $\Omega \cap B(x_{0},r)$ nor 
$\Omega^{C} \cap B(x_{0},r)$ is a Lebesgue null set for any radius $r>0$.\\

We assume that this claim would be false. Then for every $x \in \R^{n}$ there is a radius $r=r(x) > 0$ such that either
$\Omega \cap B(x,r)$ or $\Omega^{C} \cap B(x,r)$ is a Lebesgue null set. We define the sets $\Omega_{1}$ by
$$
\Omega_{1} = \{ x \in \R^{n} \ | \ \exists r > 0 \ : \ \lambda(\Omega \cap B(x,r)) = 0 \}
$$
and $\Omega_{2}$ by $\{ x \in \R^{n} \ | \ \exists r > 0 \ : \ \lambda(\Omega^{C} \cap B(x,r)) = 0 \}$. Then 
$\R^{n} = \Omega_{1} \cup \Omega_{2}$ is true. Furthermore, mind that both sets 
are disjoint. If there would exists a $x \in \Omega_{1} \cap \Omega_{2}$ then
$$
0  = \lambda(\Omega \cap B(x,r)) + \lambda(\Omega^{C} \cap B(x,r)) = \lambda \bigl( B(x,r) \bigr) > 0
$$
would follow for $r>0$ chosen appropriately small enough. So, $\Omega_{1} \cap \Omega_{2}$ must be empty. 
It is also clear that $\Omega_{1}$ and $\Omega_{2}$ are 
both open sets in $\R^{n}$. But then either $\R^{n} = \Omega_{1}$ or $\R^{n} = \Omega_{2}$ must be true since $\R^{n}$
is connected.\\

Let $\R^{n} = \Omega_{1}$ be true. We choose a compact set $K \subseteq \Omega$ such that $\lambda(K) > 0$. This is possible
since $\Omega$ is not a pure point set in $\R^{n}$. For every $x \in K$ there exists a $r = r(x) > 0$ such that 
$$
\lambda (B(x,r) \cap \Omega)=0
$$ 
since $K \subseteq \Omega \subseteq \R^{n}=\Omega_{1}$. By the compactness of $K$ we conclude that 
$$
K \subseteq B(x_{1},r_{1}) \cup \dots \cup B(x_{m},r_{m}) 
$$
holds for some $x_{1}, \dots, x_{m} \in K$ where $r_{j}$ are given by $\Omega_{1}$. But then
$$
0 < \lambda(K) = \lambda(K \cap \Omega) \leq \sum_{j=1}^{m} \lambda \bigl( B(x_{j},r_{j}) \cap \Omega \bigr) = 0
$$
is obviously a contradiction. Using the same argumentation for $\R^{n} = \Omega_{2}$ leads to a contradiction as well. 
Hence there must exist a $x_{0} \in \R^{n}$ such that neither $\Omega \cap B(x_{0},r)$ nor $\Omega^{C} \cap B(x_{0},r)$ 
is a Lebesgue null set for any radius $r>0$.
\item Let $u \in C_{c}^{\infty}(\R^{n}) \subseteq D(h)$ be a function such that $u(x_{0}) = 1$. 
So $1_{\Omega} u$ is contained in $D(h)$ as we mentioned in the very beginning of this proof. 
In particular $1_{\Omega} u$ is an element of $H^{1}(\R^{n})$ with 
$\partial_{j} (1_{\Omega} u) = 1_{\Omega} \partial_{j}u$. But this last equation needs further explanation. First remember
that $\Omega$ is not a Lebesgue null set. Therefore we choose $v \in C_{c}^{\infty}(\Omega)$. So
\begin{align*}
& \int_{\Omega} (\partial_{j} u)(x) \overline{v(x)} \ \mathrm{d}x = - \int_{\Omega} u(x) \overline{(\partial_{j} v)(x)} \ \mathrm{d}x
 =  - \int_{\R^{n}} (1_{\Omega} u)(x) \overline{(\partial_{j} v)(x)} \ \mathrm{d}x \\
& = \int_{\R^{n}} \partial_{j}(1_{\Omega} u)(x) \overline{v(x)} \ \mathrm{d}x
 = \int_{\Omega} \partial_{j}(1_{\Omega} u)(x) \overline{v(x)} \ \mathrm{d}x
\end{align*}
is true where the last equation is due to the support of $v$ in $\Omega$. Hence we see that $\partial_{j} u$ is equal to 
$\partial_{j}(1_{\Omega} u)$ in $\Omega$ again due to the fundamental lemma in the calculus of variations. Now we remember 
that $\Omega^{C}$ is also not a Lebesgue null set. We
use the similar arguments for $v \in C_{c}^{\infty}(\Omega^{C})$. So
\begin{align*}
0 & = - \int_{\R^{n}} (1_{\Omega} u)(x) \overline{(\partial_{j} v)(x)} \ \mathrm{d}x = 
\int_{\R^{n}} \partial_{j}(1_{\Omega} u)(x) \overline{v(x)} \ \mathrm{d}x \\
& = \int_{\Omega^{C}} \partial_{j}(1_{\Omega} u)(x) \overline{v(x)} \ \mathrm{d}x 
\end{align*}
holds due to the support of $v$ in $\Omega^{C}$. Therefore $\partial_{j}(1_{\Omega} u)$ is $0$ in $\Omega^{C}$. So, 
in total we have shown that $\partial_{j} (1_{\Omega} u) = 1_{\Omega} \partial_{j}u$ is actually true.\\
Since $u \in C_{c}^{\infty}(\R^{n})$ we infer that $1_{\Omega} u$ is contained in the Sobolev spaces 
$W^{1,p}(\R^{n})$ for every $p \geq 2$. Sobolev embedding theorems state that $1_{\Omega} u$ has a 
continuous representative on $\R^{n}$. We treat $1_{\Omega} u$ as a continuous function without minding mathematical rigor.\\

Remember that for $x_{0}$ we have shown that neither $\Omega \cap B(x_{0},\frac{1}{k})$ nor $\Omega^{C} \cap B(x_{0},\frac{1}{k})$
are Lebesgue null sets for every $k \in \N$. Hence we choose two sequences by
$x_{k} \in \Omega \cap B(x_{0},\frac{1}{k})$ and $y_{k} \in \Omega^{C} \cap B(x_{0},\frac{1}{k})$. Then 
$$ 
|(1_{\Omega} u)(x_{k}) -  (1_{\Omega} u)(y_{k})| = |u(x_{k})|
$$
converges to $|u(x_{0})|$ for $k \to \infty$. But since $u(x_{0})$ is equal to $1$ by the choice of $u$ this contradicts the 
continuity of $1_{\Omega} u$. Hence our assumption leads to a contradiction which proves the claim.
\end{enumerate}
\end{proof}

\section{Miscellaneous}\label{Miscellaneous}

\begin{lem}
$\int_{\frac{1}{2}\ln(2)}^{\infty} f(r)^{-1} \d r \in (0,\infty)$
\end{lem}

\begin{proof}
We use the definition of $f$ from Section \ref{IU_e^-tH}. Mind that 
$$
\frac{\d}{\d r} \left( \ln^{m-1} (r) \right)^{1-k} = \frac{1}{1-k} f_{k, m-1} (r)^{-1}
$$
holds for every $r > r_{0}$. Hence we conclude to
$$
\int_{\frac{1}{2}\ln(2)} f(r)^{-1} \d r = f_{k, m-1} ( r_{0} ) \int_{\frac{1}{2}\ln(2)}^{r_{0}} \mathrm{e}^{1-\frac{r}{r_{0}}} \d r
+ (k-1) \left( \ln^{m-1} (r_{0}) \right)^{1-k}.
$$
\end{proof}

For the following proposition please remember that $f$ and $g$ are defined in Section \ref{IU_e^-tH} and 
$\varepsilon_{t}$ is taken from Section \ref{Preliminary_Proof_IU}.

\begin{prp} \label{definition_Mt}
$\int_{2}^{\infty} \beta \bigl( \varepsilon_{t}(p) \bigr)p^{-2} \ \mathrm{d}p \in \R$
\end{prp}

\begin{proof}
First, we state that
$$
\beta \left( \varepsilon_{t}(p) \right) = 
\frac{\varepsilon_{t}(p)}{2} - \frac{n}{4} \ln( \frac{\varepsilon_{t}(p)}{2}) + \sqrt{2} \ g (\frac{1}{\sqrt{2} d \varepsilon_{t}(p)}) + C
$$ 
is given by the definition of $\beta$ in Rosen's lemma \ref{Rosen_lemma} and $\gamma$ in (\ref{gamma}).
\begin{enumerate}[i.)]
\item It is easy to see that
$
0 \leq \int_{2}^{\infty} \frac{\varepsilon_{t}(p)}{p^{2}} \d p \leq 
\int_{2}^{\infty} \frac{\varepsilon_{t}(p)}{p} \d p = t 
$ holds.
\item Furthermore, 
$
\int_{2}^{\infty} Cp^{-2} \ \mathrm{d}p = \frac{C}{2}
$
is true.
\item Now we focus on $ \ln( \frac{\varepsilon_{t}(p)}{2})$. Using the definition of $\varepsilon_{t}$ we state that
$$
 \ln \left( \frac{\varepsilon_{t}(p)}{2} \right) = K - \ln \left( f(\frac{1}{2} \ln(p) + \xi (t) ) \right)
$$
is true for a constant $K \in \R$. We use $f(r) \leq \mathrm{exp}(\mathrm{e}^{r})$ for $r \geq 1$. Then
$$
\ln \left( f(\frac{1}{2} \ln(2) + \xi (t) ) \right) \leq \ln \left( f(\frac{1}{2} \ln(p) + \xi (t) ) \right) 
\leq \mathrm{e}^{\xi (t)} p^{\frac{1}{2}}
$$
holds for $\max \{ 2, \mathrm{exp}\bigl(2(1 - \xi (t)) \bigr) \} \leq p$ due to the monotonicity of $f$. The improper 
integral of the upper bounding function is
$$
\int_{2}^{\infty} \mathrm{e}^{\xi (t)} p^{-\frac{3}{2}} \d p = \sqrt{2}\mathrm{e}^{\xi (t)}.
$$
For the lower bounding function 
$$
\int_{2}^{\infty} \frac{\ln \left( f(\frac{1}{2} \ln(2) + \xi (t) ) \right)}{p^{2}} \d p =  
\frac{\ln \left( f(\frac{1}{2} \ln(2) + \xi (t) ) \right)}{2}
$$
is true. Hence the improper integral of $\ln( \frac{\varepsilon_{t}(p)}{2}) p^{-2}$ satisfies 
$$
\frac{K}{2} - \sqrt{2}\mathrm{e}^{\xi (t)} \ \leq \ 
\int_{2}^{\infty} \ln( \frac{\varepsilon_{t}(p)}{2}) p^{-2} \d p \ \leq \ 
\frac{K}{2} -\frac{\ln \bigl( f(\frac{1}{2} \ln(2) + \xi (t) ) \bigr)}{2}. \\
$$
\item Finally we focus on the improper integral of $g (\frac{1}{\sqrt{2} d \varepsilon_{t}(p)}) p^{-2}$ over $(2,\infty)$. The equations
\begin{align*}
& \int_{2}^{\infty} g \left( \frac{1}{\sqrt{2} d \varepsilon_{t}(p)} \right) p^{-2} \d p
 = \int_{2}^{\infty} g \left( f \left( \frac{1}{2} \ln (p) + \xi (t) \right) \right) p^{-2} \d p \\
\ \\
& = \int_{2}^{\infty} \mathrm{exp} \left( \frac{1}{2} \ln (p) + \xi (t) \right) p^{-2} \d p
 = \mathrm{e}^{\xi (t)} \int_{2}^{\infty} p^{-\frac{3}{2}} \d p  = \sqrt{2} \mathrm{e}^{\xi (t)} 
\end{align*}
hold where we used $g \circ f = \mathrm{exp}$ on $\R$ which follows from $g$ being the inverse function of $f \circ \ln$ on $(0,\infty)$.
\end{enumerate}
\end{proof}

\newpage

\bibliography{references}

@article{Leinfelder81,
author = {Herbert Leinfelder and Christian G. Simader},
journal = {Mathematische Zeitschrift},
number = {1},
pages = {01--19},
publisher = {Springer},
title = {{Schrödinger Operators with Singular Magnetic Vector Potentials}},
volume = {176},
year = {1981}
}

@book{Ouhabaz05,
title = {{Analysis of Heat Equations on Domains}},
author = {El-Maati Ouhabaz},
publisher = {Princeton University Press},
address = {Princeton},
year = {2005}
}

@book{GilbargTrudinger01,
author = {David Gilbarg and Neil S. Trudinger},
title = {{Elliptic Partial Differential Equations of Second Order}},
publisher = {Springer},
address = {Berlin, Heidelberg},
year = {2001}
}

@book{Teschl09,
author = {Gerald Teschl},
title = {{Mathematical Methods in Quantum Mechanics}},
publisher = {American Mathematical Society},
address = {Providence, Rhode Island},
year = {2009}
}

@book{Davies07,
author = {Edward Brian Davies},
title = {{Heat kernels and spectral theory}},
publisher = {Cambridge University Press},
year = {2007}
}

@article{DaviesSimon84,
title = {{Ultracontractivity and the heat kernel for Schrödinger operators and Dirichlet Laplacians}},
journal = {Journal of Functional Analysis},
volume = {59},
number = {2},
pages = {335--395},
year = {1984},
author = {Edward Brian Davies and Barry Simon}
}

@article{AlziaryTakac09,
title = {{Intrinsic ultracontractivity of a Schrödinger semigroup in $\R^{N}$}},
journal = {Journal of Functional Analysis},
volume = {256},
number = {12},
pages = {4095--4127},
year = {2009},
issn = {0022-1236},
author = {Bénédicte Alziary and Peter Takáč}
}

@article{Hundertmark_Jex_Lange_2023, 
title={{Quantum systems at the brink: existence of bound states, critical potentials, and dimensionality}}, 
volume={11}, 
journal={Forum of Mathematics, Sigma}, 
author={Hundertmark, Dirk and Jex, Michal and Lange, Markus}, 
year={2023}, 
pages={e61}
}

@article{Agmon85,
author = {Shmuel Agmon},
title = {{Bounds on exponential decay of eigenfunctions of Schr{\"o}dinger operators}},
year = {1985},
publisher = {Springer Berlin Heidelberg},
pages = {1--38}
}

@article{BANUELOS1991,
title = {{Intrinsic ultracontractivity and eigenfunction estimates for Schr{\"o}dinger operators}},
journal = {Journal of Functional Analysis},
volume = {100},
number = {1},
pages = {181-206},
year = {1991},
issn = {0022-1236},
doi = {https://doi.org/10.1016/0022-1236(91)90107-G},
url = {https://www.sciencedirect.com/science/article/pii/002212369190107G},
author = {Rodrigo Bañuelos},
}
\bibliographystyle{plain}

\end{document}